\documentclass{amsart}

\usepackage{amsfonts}
\usepackage{amsthm}     
\usepackage{amsmath}    
\usepackage[mathscr]{eucal}
\usepackage{amscd}
\usepackage{setspace}
\usepackage{graphicx}
\usepackage{verbatim}   
\usepackage{color}      
\usepackage{subfigure}  
\usepackage{hyperref}   

\newtheorem{thm}{Theorem}[section]

\newtheorem{lemma}[thm]{Lemma}

\newtheorem{corollary}[thm]{Corollary}
\theoremstyle{definition}
\newtheorem{defn}[thm]{Definition}

\newtheorem{example}{Example}[section]
\numberwithin{equation}{section}
\def\BMT{\lower0.9em\hbox{\includegraphics{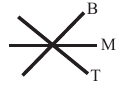}}}
\def\TMB{\lower0.9em\hbox{\includegraphics{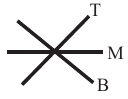}}}
\def\Msplit{\lower0.8em\hbox{\includegraphics{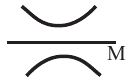}}}
\def\Tsplit{\lower1.0em\hbox{\includegraphics{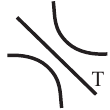}}}
\def\Bsplit{\lower1.0em\hbox{\includegraphics{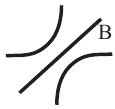}}}
\def\Ufirstsplit{\lower0.7em\hbox{\includegraphics{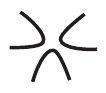}}}
\def\Usecondsplit{\lower0.7em\hbox{\includegraphics{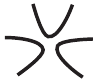}}}

\begin{document}
\title[Triple Crossing Number of Knots and Links]
{Triple Crossing Number of Knots and Links}

\date{\today}
\author[Colin Adams]{Colin Adams}
\address{Department of Mathematics and Statistics, Williams College, Williamstown, MA 01267}
\email{Colin.C.Adams@williams.edu}

\begin{abstract} A triple crossing is a crossing in a projection of a knot or link that has three strands of the knot passing straight through it.  A triple crossing projection is a projection such that all of the crossings are triple crossings. We prove that every knot and link has a triple crossing projection and then investigate $c_3(K)$, which is the minimum number of triple crossings in a projection of $K$.  We obtain upper and lower bounds on $c_3(K)$ in terms of the traditional crossing number and show that both are realized. We also relate triple crossing number to the span of the bracket polynomial and use this to determine $c_3(K)$ for a variety of knots and links. We then use $c_3(K)$ to obtain bounds on the volume of a hyperbolic knot or link. We also consider extensions to $c_n(K)$.
\end{abstract}
\maketitle

\section{Introduction}\label{S:intro} Since the study of knots began, projections of knots with finitely many transverse double points, called crossings, have played a fundamental role. The crossing number, denoted $c(K)$, is the least number of such crossings in any projection of the given knot. In this paper, we consider crossings where three strands of the knot cross, a top strand, a middle strand and a bottom strand. (See Figure \ref{triple}). The strand labelled T is on the top, M is in the middle and B is on the bottom.
We define a projection to be a triple crossing projection if the only singular points in the projection are triple crossings. See Figure \ref{trefoil} to see triple-crossing projections of the trefoil and figure-eight knot. (In fact, these are the only two nontrivial knots with projections consisting of only two triple crossings.)

\begin{figure}[h]
\begin{center}
\includegraphics[scale=0.7]{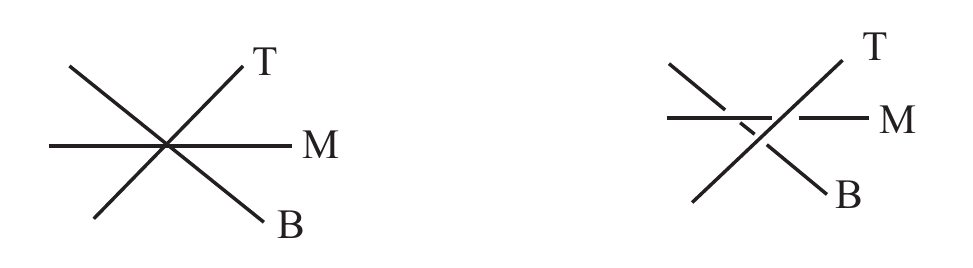}
\caption{A triple crossing and a view slightly to the side of that same crossing so we can see how the single triple crossing resolves into three double crossings.}
\label{triple}
\end{center}
\end{figure}

\begin{figure}[h]
\begin{center}
\includegraphics[scale=0.7]{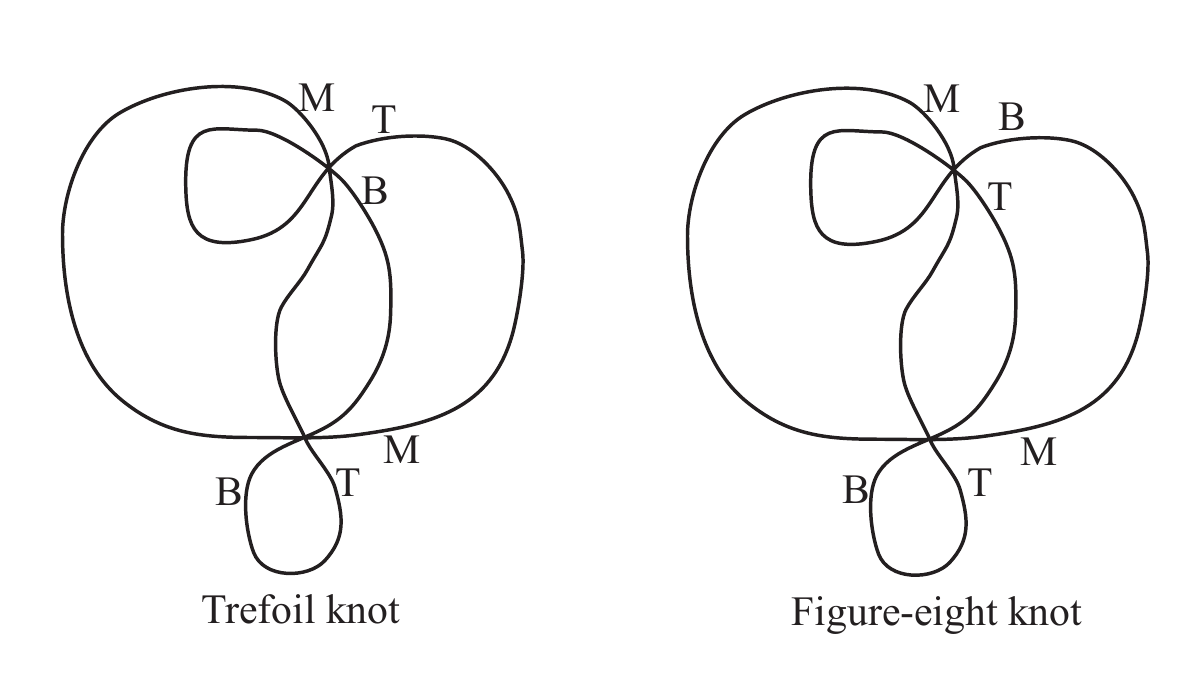}
\caption{Triple crossing projections.}
\label{trefoil}
\end{center}
\end{figure}

     In Section 2, we prove that every knot has a triple-crossing projection. Hence, we can define $c_3(K)$ to be the minimum triple crossing number for any projection of the knot $K$. We also prove that $\frac{c(K)}{3} \le c_3(K) \le c(K) -1$. We show that there exist links that realize the lower bound and that for any knot or link other than a 2-braid knot, $c_3(K) \le c(K) -2$.

     In Section 3, we obtain skein relations for the bracket polynomial applied to a triple crossing projection. We use these to show that the span of the bracket polynomial must always be at most $8c_3(K)$. This can then be used to determine exactly the triple crossing number for a substantial subset of the alternating knots and links.
     
     In Section 4, we consider generalizations to projections with only $n$-crossings for a fixed value of $n \ge 2$. We show that for every $n \ge 2$, there exists an $n$-crossing projection. Hence we can define $c_n(K)$ to be the minimum number of $n$-crossings in an $n$-crossing projection of $K$. 

     In Section 5, we consider applications to hyperbolic knots. In particular, we prove that if $K$ is hyperbolic, then $vol(S^3 -L) \leq 2(3.6638\dots) (c_3(L)-2) + 4(1.01494\dots)$.
     
Triple crossing projections were first considered in \cite{TT}, where the authors investigated triple crossing projections of graphs. They showed that most multipartite graphs have no triple crossing projection. Their own work was motivated by \cite {PT}, where so-called degenerate drawings of graphs were considered, allowing $n$-crossings of any $n$ in the same drawing of the graph.

I would like to thank T. Crawford, B. DeMeo, M. Landry, A. Lin, M. Montee, S. Park, S. Venkatesh and F. Yhee for very helpful conversations. In a subsequent paper with these co-authors (\cite{Ad}), we investigate projections of knots with just a single multi-crossing. We show that every knot and link has such a projection.
    
\section{Existence of triple crossing projections}

We begin by noting a simple procedure for turning traditional crossings, from now on called double crossings, into triple crossings. Let $P$ be a regular projection of a knot and let $C$ be a topological circle in the projection plane that intersects the projection $P$ only at crossings, and when it does intersect a crossing, it passes out the other side of the crossing. That is to say, it has two strands of the knot coming out of the crossing to either side of it. (See Figure \ref{folding}(a)). We call such a circle a {\em crossing covering circle}.Then, as in Figure \ref{folding}(b), we can take a strand of the knot coming out of one of the crossings, and stretch it around the circle, laying it on top of each of the crossings intersected by the circle. Note that we create a monogon region at the initial end of the newly stretched strand. If we choose the strand that we stretch to be the overstrand at the initial crossing, then the two strands of the mongon both pass over the third strand, so that monogon allows us to isotope the knot to eliminate one of the triple crossings that result as in Figure \ref{folding}(c). Hence,  if the circle initially intersected $n$ double crossings, we generate $n-1$ triple crossings when we perform this operation. We call this operation a {\em folding}.

\begin{figure}[h]
\begin{center}
\includegraphics[scale=0.7]{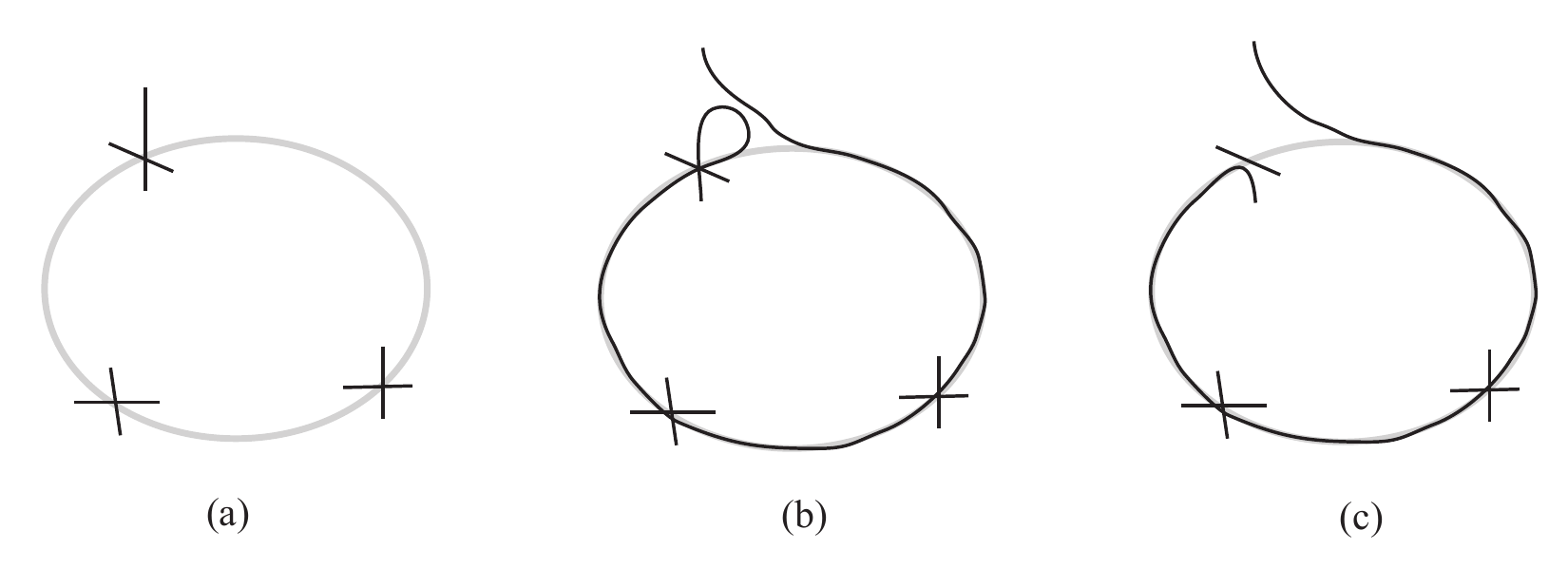}
\caption{Using a circle to turn double crossings into triple crossings via a folding}
\label{folding}
\end{center}
\end{figure}

\begin{defn} A collection of disjoint circles $\mathcal{C}$ in the projection plane is said  to be a {\em crossing covering collection} for a projection $P$ if the circles intersect the crossings as described above and every crossing intersects exactly one of the circles.
\end{defn}

In 1923, Alexander \cite{Alex} gave the first proof that every knot and link can be put in braid form. If a knot is in braid form with braid $\beta$, define $c_{\beta}$ to be the number of crossings in $\beta$ and $s_{\beta}$ to be the number of strings in $\beta$.

\begin{thm}\label{braidtheorem} Every knot and link $L$  has a triple crossing projection with no more than $c_{\beta} - s_{\beta} + 1$ triple crossings, where $c_{\beta} - s_{\beta} + 1$ can then be minimized over all braid representations $\beta$ of $L$.
\end{thm}

\begin{proof} If we put $L$ in braid form, we can easily find a crossing covering collection $\mathcal{C}$. Just take the concentric circles between the strings. Outside the braided portion, they do not intersect the knot, and inside the braided portion, they intersect the knot exactly on the crossings. The number of these circles is $s_{\beta} - 1$. Applying a folding along each of these circles yields a projection with $c_{\beta} - s_{\beta} + 1$ crossings.
\end{proof}

Although this theorem guarantees the existence of a triple crossing projection, it can be the case that the minimum number of crossings in a braid that represent $L$ can be substantially large than the crossing number of $L$. However, one can obtain a quadratic bound in crossing number for the minimal braid length, which yields a bound on triple crossing number that is quadratic in the double crossing number of the knot or link. However, a linear bound would be preferable. Hence, we prove the following theorem.

\begin{thm}Every knot or link $L$ other than a 2-braid knot has a triple crossing projection with no more than $c(K)-2$ triple crossings, and a 2-braid knot has a triple crossing projection with no more than $c(K) -1$ triple crossings.
\end{thm}

\begin{proof} We will show that any projection of $L$ has a crossing covering collection $\mathcal{C}$. Let $P$ be any projection. The existence of a crossing covering collection is independent of the choice of overstrand at each crossing. Every projection can be turned into a projection of the trivial knot or link by changing the crossings. Moreover, a projection of the trivial knot or link must be equivalent to the standard projection of a trivial knot or link via the three Reidemeister moves. The standard projection of the trivial knot or trivial link does have a crossing covering collection, but it just happens to be the empty collection. So it is enough to show that the three Reidemeister moves preserve the existence of a crossing covering collection to show that $P$ possesses a crossing covering collection. We consider each of the Reidemeister moves in turn.

In Figure \ref{ReidemeisterI}, we see that a Type I move preserves the existence of a crossing covering collection. In Figure \ref{ReidemeisterII},we see that a Type II move preserves the existence of a covering crossing projection. In Figure \ref{ReidemeisterII}(a), we see that having no circles or one circle of the crossing covering collection pass between the two strands where a Type II Reidemeister move occurs preserves the crossing covering collection. In Figure \ref{ReidemeisterII}(b), we assume that an odd number of circles pass between the two strands. Before performing the Type II move, we perform surgery on the covering circles in order that only one covering circle passes between the two strands and then we perform the Type II move. In Figure\ref{ReidemeisterII}(c), we assume that an even number of circles pass between the two strands. Before performing the Type II move, we perform surgery on the covering circles in order that no covering circle passes between the two strands and then we perform the Type II move.  In Figure \ref{ReidemeisterIII},we see that a Type III move also preserves the existence of a covering crossing projection. Hence, $P$ possesses a crossing covering collection.

\begin{figure}[h]
\begin{center}
\includegraphics[scale=0.7]{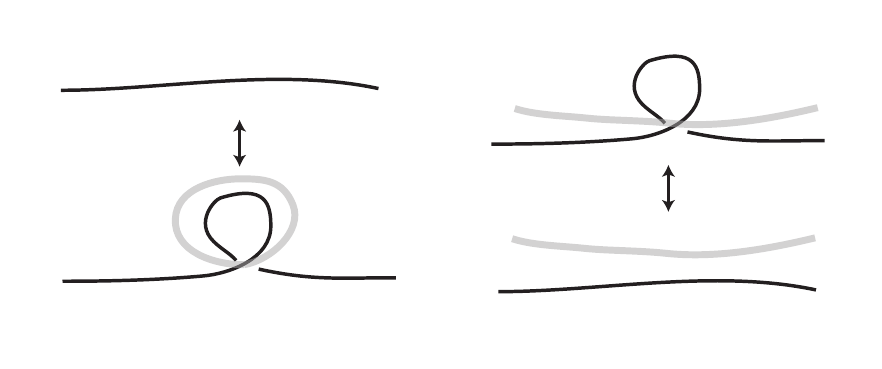}
\caption{Type I moves preserve a crossing covering collection(which appears in gray.)}
\label{ReidemeisterI}
\end{center}
\end{figure}

\begin{figure}[h]
\begin{center}
\includegraphics[scale=0.7]{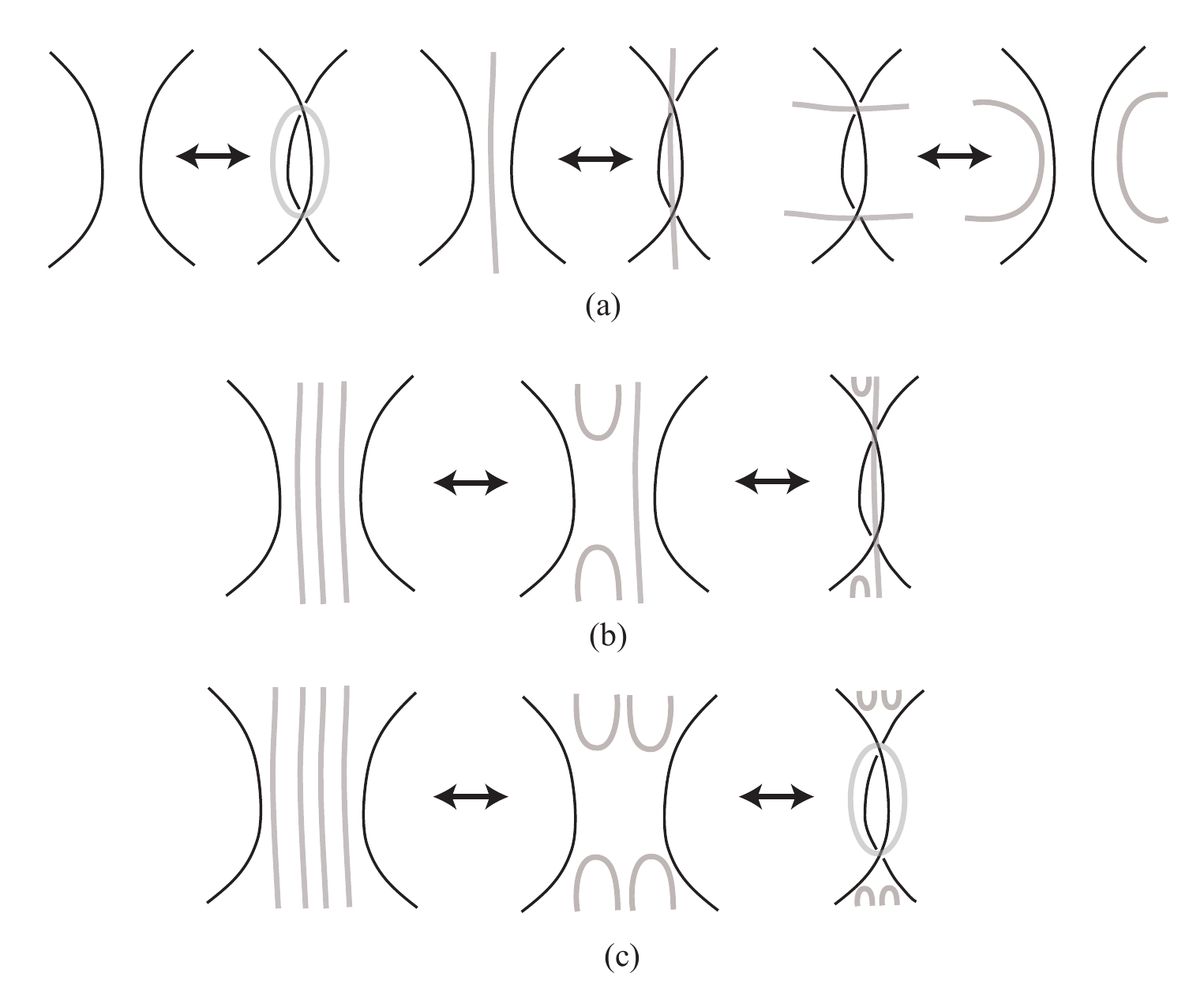}
\caption{Type II moves preserve a crossing covering collection.}
\label{ReidemeisterII}
\end{center}
\end{figure}

\begin{figure}[h]
\begin{center}
\includegraphics[scale=0.7]{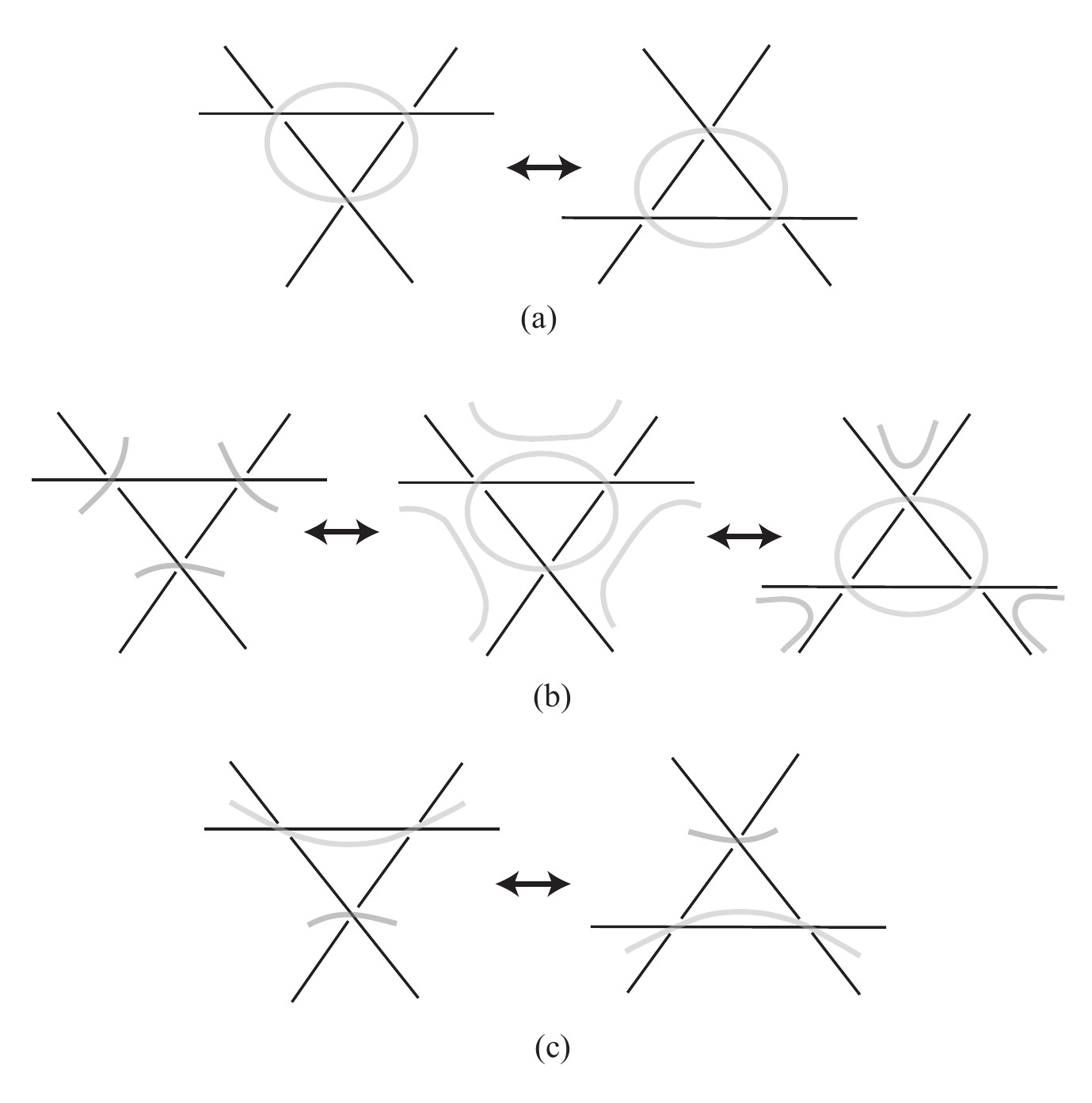}
\caption{Type III moves preserve a crossing covering collection(which appears in gray.)}
\label{ReidemeisterIII}
\end{center}
\end{figure}

We assume that all circles in a crossing covering collection intersect at least one crossing, as if not, the circle is superfluous.  As we already mentioned, with each circle in a crossing covering collection, when we fold, we lower the number of triple crossings by 1. So if there are at least two circles in our collection of crossing covering circles, the resultant number of triple crossings after folding will be at most $c(L)-2$. So suppose that there is only one circle  in the crossing covering collection and there does not exist a crossing covering collection of $P$ with more circles. So all crossings appear on $C$ and the crossings must be connected by the knot or link to the inside and outside of $C$ as appear in Figure \ref{2braid} since otherwise, we could do surgery on our single circle and split it into more than one circle. Hence, the only knots and links that have only one circle in their maximal crossing covering collection are the 2-braid knots. Note that a 2-braid link has a larger maximal crossing covering collection obtained by taking a circle around every other bigon in the projection. 
 \end{proof}

\begin{figure}[h]
\begin{center}
\includegraphics[scale=0.4]{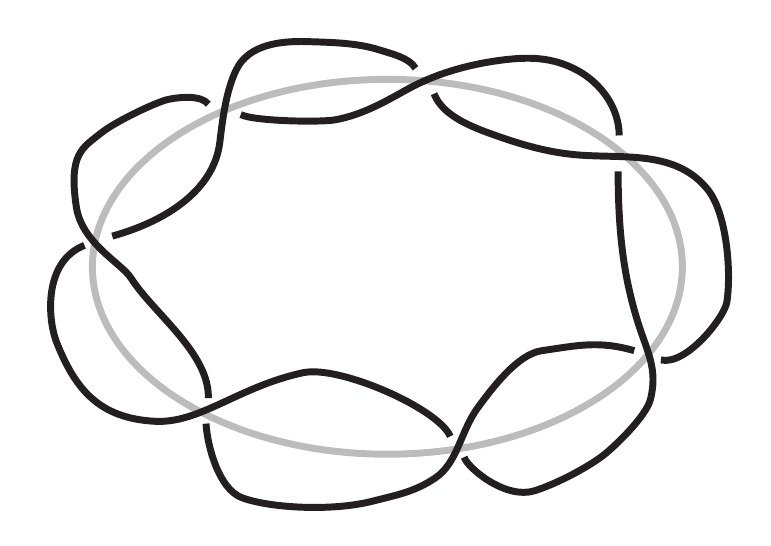}
\caption{A projection with at most one circle in its covering crossing collection.}
\label{2braid}
\end{center}
\end{figure}

One would like to prove that in fact, if $K$ is a 2-braid knot, then $c_3(K) = c(K) -1$. Also note that the trefoil and figure-eight knot realize the upper bounds given by the theorem.

 It is immediate that $\frac{c(K)}{3} \le c_3(K)$, since any triple crossing can be resolved into three double crossings. We demonstrate that there are links that realize this lower bound.
 
 \begin{figure}[h]
\begin{center}
\includegraphics[scale=0.7]{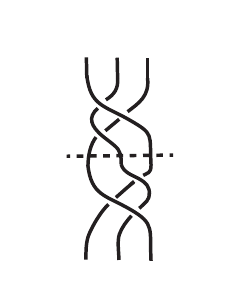}
\caption{Two triple crossings generate a full twist in the braid.}
\label{twist}
\end{center}
\end{figure}

\begin{example} \label{string} Take a 3-string braid and note that two particular triple crossings as in Figure \ref{twist}  yield one full twist in the braid. Hence by applying $m$ such pairs of triple crossings and closing the braid, we obtain a $(3m, 3)$-torus link with three components, each of which appears as an $(m, 1)$-curve on the torus. The crossing number of a $(3m,3)$-torus link is known to be $6m$, thereby yielding infinitely many examples of links that satisfy   $\frac{c(K)}{3} = c_3(K)$.
\end{example}

One might be tempted to manipulate such an example in order to obtain a knot that realizes the lower bound rather than a link, but note that a triple crossing braid will always be a link. Each triple crossing will impact three adjacent strings, and will switch the two outer ones while fixing the middle one. Hence, if the strings are labelled $1, 2, \dots, n$ from left to right, a triple crossing braid can never change the parity of a string. There will always be at least two components in the resultant link.

\begin{example} Consider the link that appear in Figure \ref{cubelink}. Note that each component has linking number 1 with four other components. Hence, as there must be two crossings between any pair of linked components, this link must have a crossing number of at least 24. As this is the number of crossings in this projection, we obtain a crossing number of 24. However, this projection also demonstrates that this link has a triple crossing projection with  8 triple crossings, thereby yielding another example with $\frac{c(K)}{3} = c_3(K)$.

\begin{figure}[h]
\begin{center}
\includegraphics[scale=0.7]{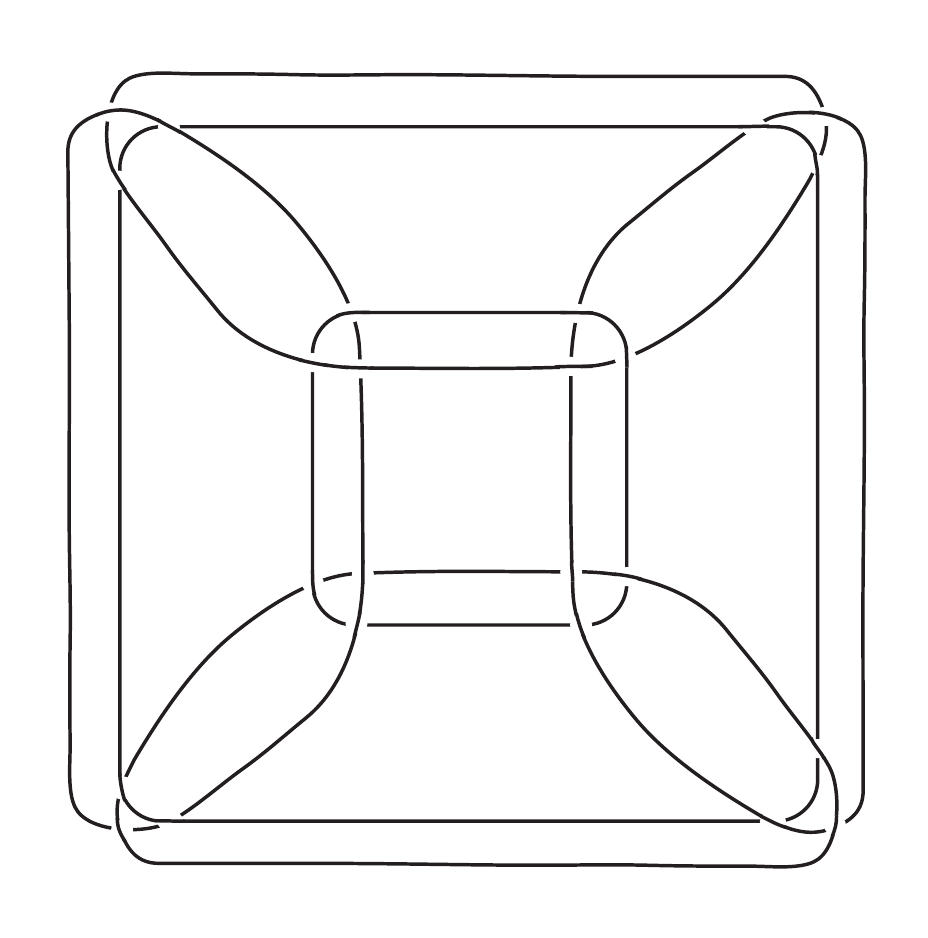}
\caption{A link that satisfies $\frac{c(K)}{3} = c_3(K)$.}
\label{cubelink}
\end{center}
\end{figure}

A second link, based on a dodecahedron, can be similarly constructed. Each of the twelve components links with five others, which by the same type of argument, leads to $\frac{c(K)}{3} = c_3(K)$ once again. Additional examples can be constructed from the 1-skeleta of other trivalent planar graphs.
\end{example}

\section{Triple crossing number and the bracket polynomial}

In this section, we consider what the bracket polynomial can tell us about triple crossing number. Each triple crossing in a triple crossing projection will either be a BMT crossing or a TMB crossing when the labels are read clockwise. When we resolve the single triple crossing into three double crossings and then split those crossings each in the two ways possible, we obtain eight possibilities. These subdivide into five equivalence classes, which we describe as follows. The T-split, B-split and M-split are the states obtained by splitting open the triple crossing so the T, M or B strand goes straight through respectively and the remaining strands perform a U-turn to either side of that strand. The remaining two splits are called U-splits, and appear as in the statement of the following theorem.

\begin{thm} \label{skein}\Large{$< \BMT> = A^3<\Msplit>+ A(< \Ufirstsplit> +<\Usecondsplit >) + \\A^{-1} (<\Bsplit> + <\Tsplit>)$
\vspace{.5in} \\
                              $ <\TMB> = A^{-3}<\Msplit > + A^{-1} (<\Ufirstsplit> +<\Usecondsplit>) + \\A (<\Bsplit> + <\Tsplit >)$}

\end{thm}

\begin{proof} This follows immediately by resolving each of the particular triple crossings into its three double crossings as in Figure \ref{triple}, and then applying the skein relation for the bracket polynomial  applied to double crossings (cf. \cite{Kau}) to split each of the crossings and clean up the result.
\end{proof}

\begin{lemma} \label{change} Given the bracket polynomial term associated to a particular state $S$:

\begin{enumerate} 
\item Changing from an M-split in a BMT crossing cannot increase the highest power of $A$. 
\item Changing from an M-split in an TMB crossing cannot decrease the lowest power of $A$.
\item Changing from a B-split or T-split to the M-split or a U-split  in a BMT crossing cannot decrease the lowest power of $A$. 
\item Changing from a B-split or T-split to the M-split or a U-split  in an TMB crossing cannot increase the highest power of $A$.
\end{enumerate}
\end{lemma}

\begin{proof} By Theorem \ref{skein}, changing an M-split in a BMT crossing to either a B-split or a T-split decreases the power of A multiplying the resulting bracket polynomial by 4. But the number of components in the state increases by at most 2, and therefore the additional $(-A^2 -A^{-2})^2$ that comes from the potential increase in the number of components can at most offset the loss of 4 in the power of A. Hence the new state has highest power no larger than the original. If we change an M-split in a BMT-crossing into a U-split, the number of components can change by at most one, and by Theorem \ref{skein}, the power of $A$ multiplying the resulting polynomial goes down by 2. Hence, again, the additional $(-A^2 -A^{-2})$ that comes from the potential increase in the number of components can at most offset the loss of 2 in the power of A. Hence the new state has highest power no larger than the original. A similar argument applies to prove that changing an M-split in a TMB crossing cannot decrease the lowest power of $A$.

Suppose now that we change from a B-split or T-split to the M-split in a BMT crossing.  This increases the power of $A$ multiplying the resulting polynomial by 4, and again changes the number of components of the state by at most 2. Hence, the additional $(-A^2 -A^{-2})^2$ that comes from the potential increase in the number of components can at most offset the increase of 4 in the lowest power of A. Hence the new state has lowest power at least as large as the original state. If we change the B-split or T-split to a U-split in a BMT crossing,  the number of components can change by at most one, and by Theorem \ref{skein}, the power of $A$ multiplying the resulting polynomial goes up by 2. Hence, again, the additional $(-A^2 -A^{-2})$ that comes from the potential increase in the number of components can at most offset the gain of 2 in the power of A. Hence the new state has lowest power no lower than the original. A similar argument applies to prove that changing a B-split or T-split in an TMB crossing cannot increase the highest power of $A$.
\end{proof}

\begin{thm} \label{span} For any nontrivial knot or non-splittable link, span$(<K>) \leq 8 c_3(K)$.
\end{thm}

\begin{proof}Let $P$ be a triple crossing projection of $K$ with $c_3(K)$ triple crossings. For convenience, let $t = c_3(K)$.  Each crossing is either a BMT or a TMB crossing. Define a state to be $s_{max}$ if it is obtained by splitting each BMT crossing as an M-split and each TMB crossing as either a B-split or T-split depending on which of those individual choices maximizes the number of component circles in the state. Similarly, define the state $s_{min}$ to be the state obtained by splitting each TMB crossing as the M-split and splitting each BMT crossing as either a T-split or B-split in order to again maximize the number of components in the state. 

We first prove that there is no state with a higher power of $A$ than $s_{max}$. Any other state is obtained from  $s_{max}$ by changing some number of the M-splits of the BMT crossings and some number of the B and T-splits of the TMB crossings. We consider any such change individually. But Lemma \ref{change} and the fact we have already chosen the B and T-splits of the BMT crossings to maximize the number of components imply that we cannot change any splits and increase the highest exponent of $A$.  Hence $s_{max}$ achieves the highest power of $A$ of any state. Similarly, one can prove that $s_{min}$ achieves the lowest power of any state. 

Given a projection $P$, define $M_P$ to be the highest exponent of $A$ in the polynomial associated to $s_{max}$ and  $m_P$ to be the lowest exponent of $A$ in the polynomial associated to $s_{min}$.

Let $|BMT|$ and $|TMB|$ be the number of BMT and TMB crossings in $P$ respectively. So  $|BMT| + |TMB|= t$. 
Then according to the skein relation,  

$$M_P = 3|BMT| + |TMB| +2|s_{max}|-2$$

$$m_P= -|BMT| -3|TMB| -(2|s_{min}| -2)$$

Hence, \\

\begin{align}span(<K>) \leq M_P- m_P &= 4|BMT| + 4|TMB| + 2(|s_{max}| + |s_{min}|) - 4 \notag \\
 &= 4t+ 2(|s_{max}| + |s_{min}|) - 4 \notag
 \end{align}
 
 Hence, it suffices to prove that $|s_{max}| + |s_{min}| \leq 2t + 2$.

But, considering each of the splits to create $s_{max}$ or $s_{min}$ individually, each split increases the number of components of the diagram by 0, 1 or 2. After all splits are performed to obtain one of these two states, the total number of resulting components is the number of circles in the state. But if an M-split of a BMT crossing increases the number of components of the $s_{max}$ state by 0, 1 or 2, the corresponding B or T-split of that crossing increases the number of components of the $s_{min}$ state by 2,1 or 0 components respectively. Similarly, if a B or T-split of a TMB crossing increases the number of components of the $s_{max}$ state by 0, 1 or 2, the corresponding M-split of that crossing increases the number of components of the $s_{min}$ state by 2,1 or 0 components respectively. 

If there is a single triple crossing, $|s_{max}|$ is 1,2 or 3, and $|s_{min}|$ is 3, 2 or 1 respectively, so $|s_{max}| + |s_{min}| \leq 4$.  Every subsequent triple crossing can increase the number of components of $|s_{max}|$ by 0, 1 or 2 components, but then it increases the number of components of $|s_{min}|$ by 2,1 or 0 components respectively. Hence, $|s_{max}| + |s_{min}| \leq 2t + 2$.  

\end{proof}

\begin{corollary} \label{crossing} For any nontrivial alternating knot or link, $c_3(K) \geq \frac{c(K)}{2}$.
\end{corollary}

\begin{proof} Results of \cite{Kau}, \cite{Muras} and \cite{Th} yield the fact that for a reduced alternating knot, \\ span$(<K>) = 4c[K]$. Hence, by Theorem \ref{span}, $4c(K) \leq 8 c_3(K)$, yielding the result.
\end{proof}

Given a projection of a knot, define  a {\it maximal bigon chain} to be a sequence of bigonal complementary regions touching end-to-end that is as long as possible. The {\it crossing length} of such a bigon chain is the number of crossings it contains. Note that a crossing that does not lie on a bigon is considered to be a maximal bigon chain of crossing length 1. The number of such maximal bigon chains in a projection $P$ is called the {\it twist number} of the projection, denoted $tw(P)$. 

We say a knot projection or a portion of a knot projection satisfies the {\it even bigon chain condition} if every maximal bigon chain in it has even crossing length. Note that this implies that every crossing appears in a nontrivial bigon sequence. Specific examples include the two-component 2-braid links and the subset of twist knots $4_1,6_1,8_1,\dots$.  The triple crossing number of knots that possess an alternating  projection that satisfies the even bigon chain condition can be determined exactly. 

\begin{corollary} \label{blob} Let K be an alternating knot or link with a reduced alternating projection that satisfies the even bigon chain condition. Then $c_3(K) = \frac{c(K)}{2}$. 
\end{corollary}

\begin{proof} Any bigon can be turned into a triple crossing by twisting as in Figure \ref{bigon}. Hence, by starting at one end of each maximal bigon chain and twisting every other bigon, we can turn the chain into a sequence of triple crossings, one for each pair of crossings. Hence $c_3(K) \leq \frac{c(K)}{2}$. Corollary \ref{crossing} completes the proof.
\end{proof}

\begin{figure}[h]
\begin{center}
\includegraphics[scale=0.7]{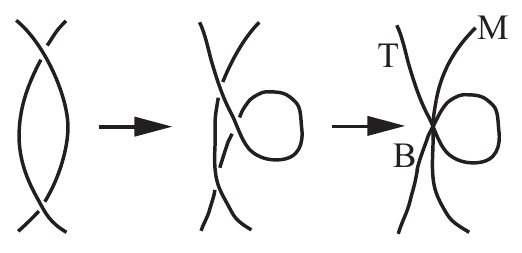}
\caption{Transforming a bigon into a triple crossing.}
\label{bigon}
\end{center}
\end{figure}

In fact, we can also determine the exact triple crossing number for additional alternating knots. Suppose that in a reduced alternating projection $P$  of a knot or link $K$, we can find a crossing covering circle that crosses exactly three crossings, and the remainder of the projection when these three crossings are excluded satisfies the even bigon chain condition. Then by folding along the crossing covering circle, we obtain two triple crossings. Together with the triple crossings obtained from the rest of the projection, we have $\frac{c+1}{2}$ triple crossings. Since the original projection is alternating with $c(K)$ crossings, we know that $c_3(K) \geq \frac{c(K)}{2}$. But the fact $c(K)$ is odd and $c_3(K)$ is an integer implies $c_3(K) \geq \frac{c+1}{2}$. Thus, we have proved the following corollary.

\begin{corollary} \label{new} Let $K$ be an alternating knot or link with a reduced alternating projection such that there exists a crossing covering circle through three crossings and the remainder of the projection satisfies the even bigon chain condition. Then $c_3(K) = \frac{c(K)+1}{2}$.
\end{corollary}

This result applies to all the twist knots not covered by Corollary \ref{blob}, for instance. Between the two types of projections encompassed by the above two corollaries, we can exactly determine the triple crossing number of the following prime alternating knots of nine or fewer crossings: $3_1$, $4_1$, $5_2$, $6_1$, $7_2$, $7_4$, $7_6$, $7_7$,  $8_1$, $8_3$, $8_{12}$, $9_2$, $9_5$, $9_8$, $9_{12}$, $9_{14}$, $9_{15}$, $9_{19}$, $9_{21}$, $9_{25}$, $9_{35}$, $9_{37}$, $9_{39}$ and $9_{41}$. Note also that these corollaries imply that triple crossing number is additive under composition when either both the factor knots have alternating projections that satisfy the even bigon chain condition, or one has an alternating projection that satisfies the even bigon chain condition and the other has a projection that satisfies the hypotheses of Corollary \ref{new}.

\section{n-crossing projections}

\begin{defn} Given a projection of a knot or link,  a {\em multi-crossing}  or {\em $n$-crossing} is a singular point corresponding to $n$ strands of the knot all crossing at a single point such that each strand passes straight through the crossing. The $n$ strands can be numbered 1 to $n$ according to the order they occur from topmost to bottommost as they pass through the crossing.
\end{defn}

\begin{defn} Let $L$ be a knot or link. An {\em $n$-crossing projection} of $L$ is a projection such that the only singularities are $n$-crossings. The {\em $n$-crossing number} of L, denoted $c_n(L)$,  is the least number of crossings in an $n$-crossing projection of $L$.
\end{defn}

The following theorem shows that the $n$-crossing number of any knot or link is defined.

\begin{thm} Given any integer $n \ge2$, and any knot or link $L$, there exists a projection of $L$ with only $n$-crossings.
\end{thm}

\begin{proof} We already know this for $n = 2, 3$. For $n=4$, we can take a double-crossing projection, and replace each double crossing with a 4-crossing as in Figure \ref{quadruple}. A similar construction works to show that there exists an $n$-crossing projection for any even $n$.

     In the case of $n$ odd, we use the following trick. We know that any double crossing projection has a crossing covering collection of circles. We first convert each crossing into an $(n-1)$-crossing projection as in the previous paragraph. We then use folding on the crossing covering collection to obtain an $n$-crossing projection.
 \end{proof}
     
 \begin{figure}[h]
\begin{center}
\includegraphics[scale=0.7]{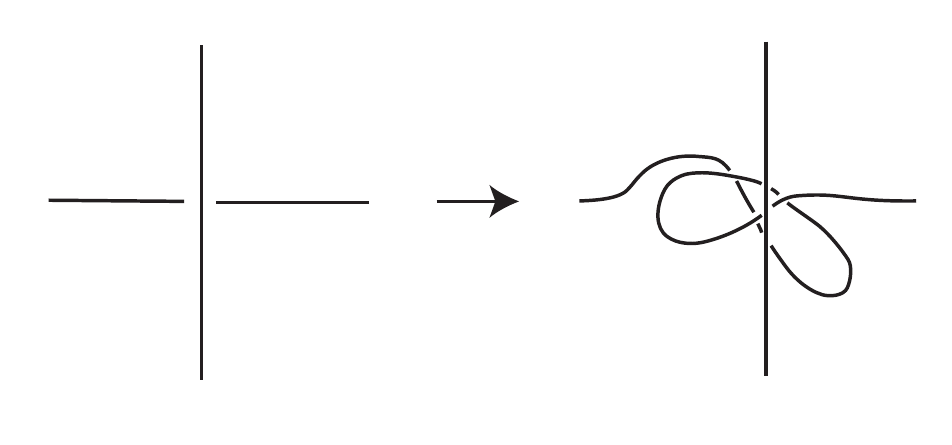}
\caption{Converting a double crossing into a quadruple crossing.}
\label{quadruple}
\end{center}
\end{figure}

     Just as we saw for triple crossing number, but adjusted for n-crossings, we have the following inequalities.
     
     $$ \frac{2c(L)}{n(n-1)} \le c_n(L) \le c(L) -1$$

Just as in Example \ref{string} when we created a 3-string braid that yielded an example with  $c(L) = 3 c_3(L)$, we can create an $n$-string braid that satisfies $c(L) = \frac{n(n-1)}{2}c_n(L)$.

\section{Triple Crossing Number and Hyperbolic Volume}

Knots fall into three disjoint categories: torus knots, satellite knots (including composite knots) and hyperbolic knots. This fact, proved by W. Thurston, has revolutionized knot theory. A knot is hyperbolic if its complement $S^3-K$ carries a hyperbolic metric. Such a metric is uniquely determined and hence, the hyperbolic volume of its complement becomes an invariant that can be used to distinguish it from other knots. 

In a variety of papers, there have been attempts to determine bounds on hyperbolic volume from projections of knots. In particular, in \cite{Lack}, a bound was given in terms of the twist number of a projection of the link, as defined in the last section. 

In an appendix by Ian Agol and Dylan Thurston to \cite{Lack}, it was proved that if a hyperbolic knot has diagram $D$, then $vol(S^3 -L) \leq 10 v_0 (tw(D)-1)$ where $v_0$ is the volume of an ideal regular tetrahedron, approximately 1.01494. Note that in the case a minimal crossing knot diagram has no bigons, this becomes  $vol(S^3 -L) \leq 10 v_0 (c(K)-1)$. 

There is also a second bound that follows from a construction of  D. Thurston. One places an octahedron at each crossing with its top vertex on the overstrand and its bottom vertex on the understrand as in Figure \ref{octahedron}. 

\begin{figure}[h]
\begin{center}
\includegraphics[scale=0.7]{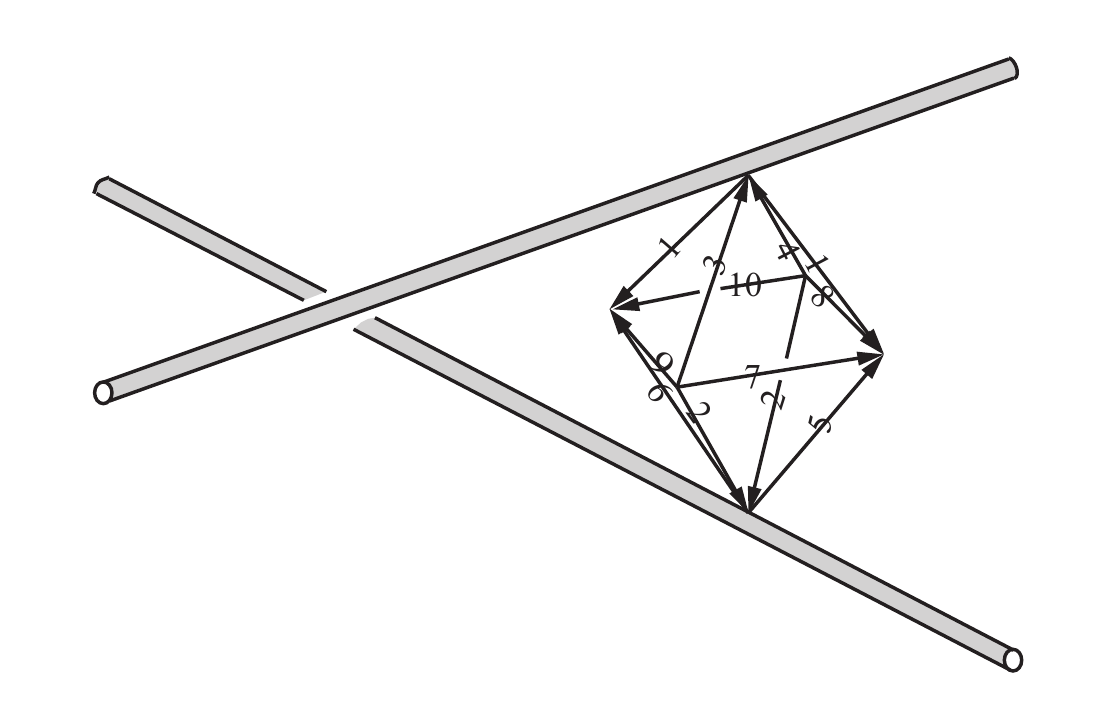}
\caption{Placing an octahedra between each crossing.}
\label{octahedron}
\end{center}
\end{figure}

Then, as per the edge labellings in Figure \ref{octahedron}, one pulls two of the opposite remaining vertices up to meet above the crossing at a point denoted $+\infty$, thereby identifying two edges on the octahedron, and pulls the remaining two vertices down to meet below the crossing at a point denoted $-\infty$, again identifying two edges of the octahedron.Then one can glue together the faces on the octahedra at the various crossings in order to fill the complement of the knot. If the two vertices at $\pm \infty$ are included, we have a decomposition of the knot complement into octahedra with some finite vertices and some ideal vertices. This construction has been used to attempt to prove Kaeshaev's Volume Conjecture for various categories of knots. See \cite{Ohn} and \cite{Yok} for instance. However, one also obtains an upper bound for the hyperbolic volume of the knot or link complement since  an octahedron in hyperbolic 3-space has  volume at most the volume of an ideal regular octahedron, which is approximately $v_{oct} = 3.6638\dots$. So $vol(S^3 -L) \leq  c(K) v_{oct}$. In fact, this bound can be improved.

\begin{thm}\label{volumethm} Let $K$ be a hyperbolic knot with $c(K) \geq 5$ crossings. Then \\ $vol(S^3 -L) \leq  (c(L) - 5) v_{oct} + 4 v_0$.
\end{thm}

\begin{proof}  Take the decomposition of the complement into octahedra, as described above. By pulling the vertices at $+\infty$ and $-\infty$ to the knot, one obtains an ideal polyhedral decomposition of the knot complement into ideal octahedra and ideal tetrahedra. As in \cite{Ohn}, we collapse two faces that are shared between two octahedra, both of which occur at the crossings at the endpoints of an edge of the projection. For any projection of five or more crossings, there always exists an edge between two faces, the first of which is an n-gon with $n \geq 3$ and the second of which is an n-gon with $n \geq 4$. The two octahedra at the ends of this edge flatten into triangles. Each of the other octahedra around these two faces collapse down to two tetrahedra, and two more octahedra on the continuation of this edge in either direction also  collapse down to two tetrahedra. However, if  there are bigons on the boundary of the two regions, these last two tetrahedra might be identified with some of the others being collapsed. In the worst case of two such bigons,  one octahedron collapses to two tetrahedra, and two others collapse to one, yielding our bound.
Thus one obtains the improved bound.
\end{proof}

We now consider the case of a knot or link in a triple crossing projection.

\begin{thm} \label{tripleoct}Let $L$ be a knot or link in $S^3$. Then the complement of $L$ can be decomposed into $2 c_3(L)$ octahedra with vertices that are ideal or occurring at $\pm \infty$.
\end{thm}

\begin{proof} At each triple crossing we again insert an octahedron, with its top vertex on the overstrand and its bottom vertex on the understrand. Then it is intersected by the middle strand which passes through a horizontal edge, through the interior and out the opposite horizontal edge, as in Figure \ref{octahedron3}. This splits each of these two horizontal edges into two edges. We add four additional edges to the exterior of the octahedron, each going over or under one of the two points where the middle strand punctures the exterior of the octahedron. We then collapse the intersection of the middle strand with the octahedron down to a point, and then split the collapsed octahedron into two octahedra, as in Figure \ref{octahedron6}. These pairs of octahedra at each triple crossing can then be glued together along faces to fill the entire knot complement. 

\begin{figure}[h]
\begin{center}
\includegraphics[scale=0.7]{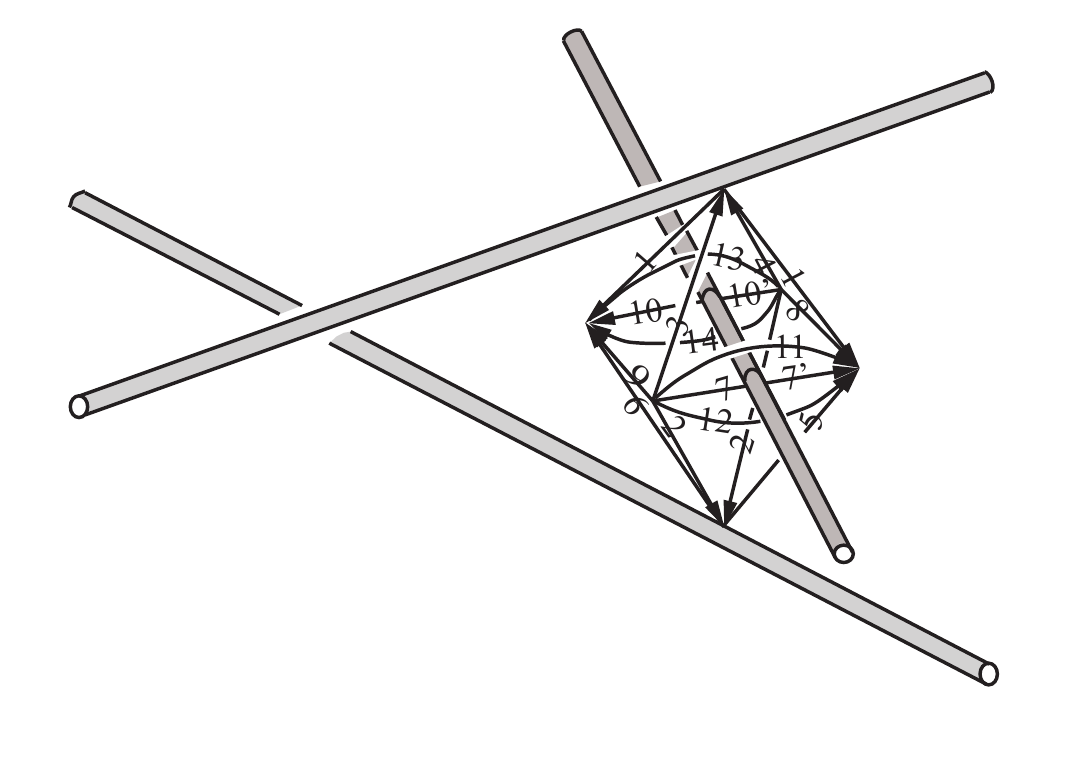}
\caption{The middle strand intersects the octahedron at each triple crossing.}
\label{octahedron3}
\end{center}
\end{figure}

\begin{figure}[h]
\begin{center}
\includegraphics[scale=0.7]{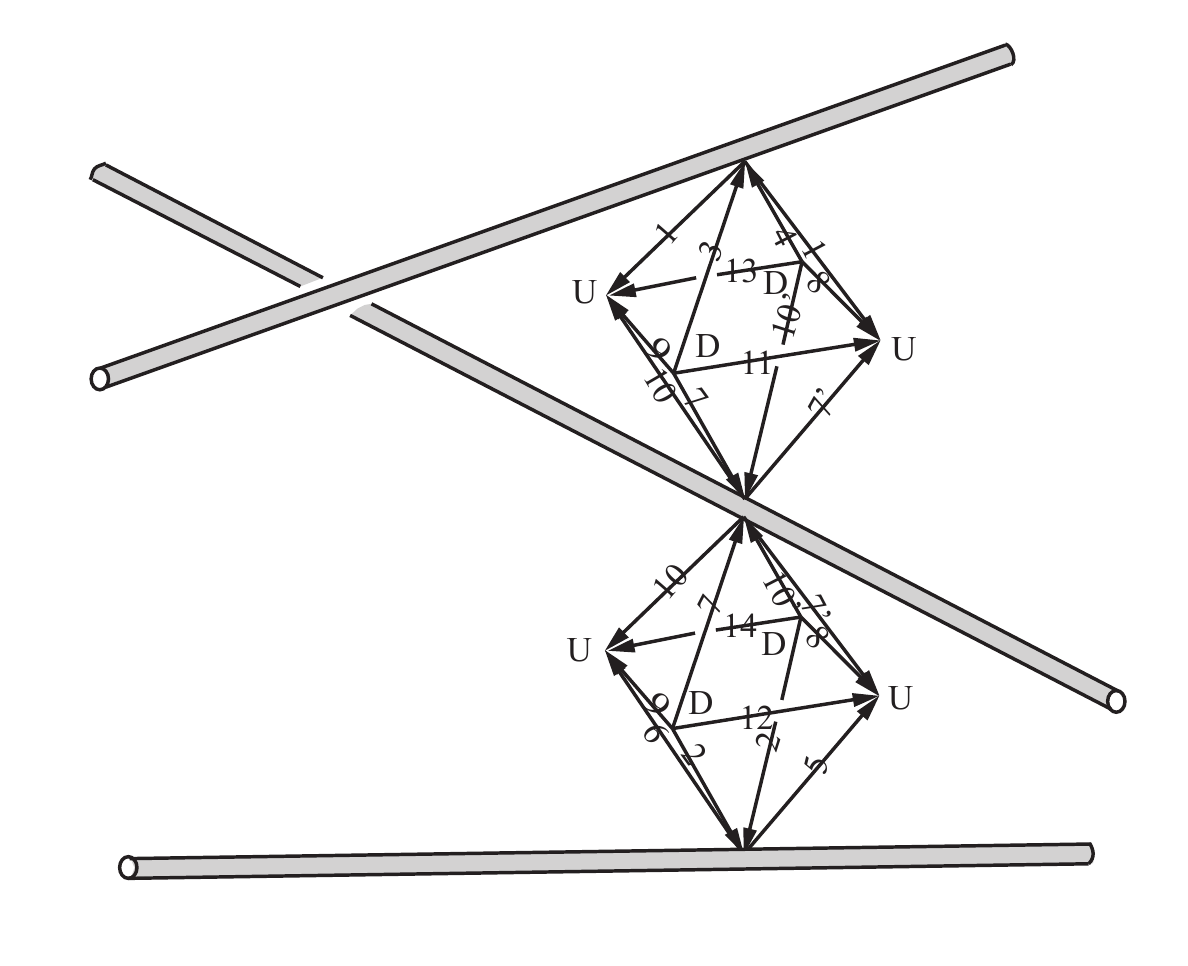}
\caption{Collapse the intersection of the middle strand with the octahedron and cut open  to obtain two octahedra.(U and D denote vertices that are pulled up and down, respectively.)}
\label{octahedron6}
\end{center}
\end{figure}

\end{proof}

\begin{corollary} Let $L$ be a hyperbolic knot or link in $S^3$. Then, $vol(S^3 -L) \leq 2 v_{oct} (c_3(L)-2) + 4 v_0$.
\end{corollary}

\begin{proof} We begin with the octahedral decomposition given by Theorem \ref{tripleoct}. As in the proof of Theorem \ref{volumethm}, we  collapse two faces shared by two octahedra in order to identify the vertices at $\pm \infty$ with the knot. These two octahedra collapse to triangular faces. There exist at least two other octahedra that share the edge types that are collapsed.  Hence those two octahedra become pairs of tetrahedra, which each have a volume of at most $2v_{0}$.
\end{proof}

Note that this bound is stronger than the bound one would obtain by resolving each triple crossing into three double crossings and then placing an octahedron at each of those crossings.

\end{document}